\documentclass[leqno]{amsart}
\date{August 19, 2014}
\usepackage{amssymb}
\usepackage[dvipdfmx]{graphicx}
\usepackage{amsthm}
\theoremstyle{plain}
 \newtheorem{introthm}{Theorem}
 
 \newtheorem{introcor}[introthm]{Corollary}

 \newtheorem{theorem}{Theorem}
 \newtheorem{proposition}[theorem]{Proposition}
 \newtheorem{lemma}[theorem]{Lemma}
 \newtheorem*{lemma*}{Lemma}
 \newtheorem{fact}[theorem]{Fact}
 
\theoremstyle{definition}
 \newtheorem{definition}[theorem]{Definition}
\theoremstyle{remark}
 \newtheorem{remark}[theorem]{Remark}
 
\numberwithin{equation}{section}

\newcommand{\op}{\operatorname}

\newcommand{\R}{\boldsymbol{R}}

\newcommand{\GL}{\op{GL}}
\newcommand{\SO}{\op{SO}}
\newcommand{\M}{\op{M}}
\newcommand{\Cusp}{\mathcal{C}}
\newcommand{\F}{\mathcal{F}}
\renewcommand{\phi}{\varphi}

\newcommand{\pmt}[1]{{\begin{pmatrix} #1  \end{pmatrix}}}

\title{%
    Isometric deformations of cuspidal edges
}
\author[K. Naokawa]{Kosuke Naokawa}
\author[M. Umehara]{Masaaki Umehara}
\author[K. Yamada]{Kotaro Yamada}

\thanks{%
  The first author was partly supported by 
  the Grant-in-Aid for JSPS Fellows.
  The second and third authors  were 
  partially supported by Grant-in-Aid for 
  Scientific Research (A) No.~262457005, 
  and Scientific Research (C) No.~26400006,
  respectively, 
  from the Japan Society for the Promotion of Science.}

\address{%
  Department of Mathematics \endgraf
  Faculty of Science \endgraf
  Kobe University \endgraf
  1-1, Rokkodaicho, Nada, Kobe, 657-8501\endgraf
  Japan
  }
\email{naokawa@port.kobe-u.ac.jp}

\address{
  Department of Mathematical and Computing Sciences \endgraf
  Tokyo Institute of Technology \endgraf
  2-12-1-W8-34, O-okayama, Meguro \endgraf
  Tokyo 152-8552 \endgraf
  Japan
  }
\email{umehara@is.titech.ac.jp}

\address{%
  Department of Mathematics \endgraf
  Tokyo Institute of Technology \endgraf
  2-12-1-H-7,  O-okayama, Meguro\endgraf
  Tokyo 152-8551\endgraf 
  Japan
}%
\email{kotaro@math.titech.ac.jp}

\begin{document}
\maketitle
\begin{abstract}
 Along  cuspidal edge singularities on a given surface in 
Euclidean 3-space $\R^3$,
 which can be parametrized by a regular space curve $\hat\gamma(t)$,
 a unit normal vector field $\nu$  is well-defined as 
 a smooth vector field of the surface.
 A cuspidal edge singular point is  called {\it generic\/} if the 
 osculating plane of $\hat\gamma(t)$ is not orthogonal to $\nu$. 
This genericity is equivalent to
 the condition  that its limiting normal curvature 
 $\kappa_\nu$ takes a non-zero value.
 In this paper, we show that a given generic 
 (real analytic) cuspidal edge $f$ 
 can be isometrically  deformed  preserving $\kappa_\nu$
 into a cuspidal edge whose singular set
 lies in a plane.
 Such a limiting cuspidal edge is uniquely determined from the initial
 germ of the cuspidal edge.
\end{abstract}

\section*{Introduction}
Let $\Sigma^2$ be a 2-manifold.
A singular point $p\in \Sigma^2$ of a $C^\infty$-map germ 
$f:(\Sigma^2,p)\to\R^3$  is a {\em cuspidal edge}\/
if $f$ at $p$ is right-left equivalent to  
$(u,v)\mapsto(u,v^2,v^3)$ 
at the origin.
Recently, the differential geometry of co-rank one singularities
(including cuspidal edges)
on surfaces was discussed by several geometers
(\cite{dt,fh,ggs,HHNUY,MB, Oset-Tari,SUY, faridtari}). 
In particular, in \cite{HHNUY}, isometric deformations
of a special class of cross caps were discussed.
Relating to this, Martins-Saji \cite{MS} defined 
several differential geometric invariants on  cuspidal edges, 
and gave geometric meanings for them.
Moreover, it was shown in \cite{MSUY} that
the limiting normal curvature $\kappa_\nu$ 
defined in \cite{SUY} (cf.\ \eqref{eq:knu}) is closely related to 
the behavior of  Gauss maps around cuspidal edges.

On the other hand, the proof of the classical Janet-Cartan 
theorem on the local existence of isometric embeddings 
of real analytic Riemannian $n$-manifolds 
into the Euclidean space $\R^{n(n+1)/2}$ yields that
any generic regular surface in $\R^3$ has a non-trivial
family of isometric deformations.
So it is natural to expect the existence of such 
non-trivial isometric deformations 
for surfaces with singularities.
As shown in \cite{HHNUY} and \cite{MSUY},  
certain classes of ruled cross caps and cuspidal edges
admit non-trivial isometric deformations, respectively.
However, general cases have not been discussed yet. 

Along cuspidal edge singularities of a given $C^\infty$-map germ%
\footnote{%
  Though the definitions here are for $C^{\infty}$-maps,
  we consider only real analytic map germs in this paper
  since we shall apply the Cauchy-Kovalevskaya theorem.
}%
$f:(\Sigma^2,p)\to \R^3$, the unit normal vector field $\nu$ 
is well-defined as a smooth vector field of the surface.
Let $\gamma(t)$ be a regular curve in $\Sigma^2$ 
satisfying $\gamma(0)=p$ as a 
parametrization of cuspidal edge singularities of the map $f$.
We call $\gamma(t)$ the {\it singular curve\/} of $f$.
We set
\[
     \hat \gamma(t):=f\circ \gamma(t),
\]
which is a regular space curve. 
Let $\kappa_s(t)$ be the singular curvature function 
along the curve $\gamma(t)$ (cf.\ \cite[(1.7)]{SUY}),
and $\kappa_\nu(t)$ the limiting normal curvature
along $\gamma(t)$ (cf.\ \cite[(3.11)]{SUY} and \eqref{eq:knu}). 
Then the curvature function of $\hat \gamma(t)$ 
as a space curve is given by
(cf.\ \cite{MSUY})
\begin{equation}\label{eq:kappa}
   \kappa(t)=\sqrt{\kappa_s(t)^2+\kappa_\nu(t)^2}.
\end{equation}
In \cite{SUY}, \cite{MS} and \cite{MSUY},
the singular curvature function $\kappa_s(t)$ and
the limiting normal curvature function $\kappa_\nu(t)$
are considered as geometric invariants along
cuspidal edge singularities, as well as 
the curvature function
$\kappa(t)$ and the torsion function $\tau(t)$ 
of $\hat \gamma(t)$.
The relationships amongst 
$\kappa_s$, $\kappa_\nu$ and $\tau$ are
given in \cite{MS}.

An invariant $I$ of map germs at $p \in \Sigma^2$ is called 
\textit{intrinsic}
if it is determined only by the first fundamental form
(cf. \cite{HHNSUY}).
The singular curvature $\kappa_s$ is a typical intrinsic 
invariant of a cuspidal edge singularity.
In \cite{MSUY}, the cuspidal 
curvature $\kappa_c$ at a given cuspidal edge singular
point $p$ is defined. Let $\Pi$ be the plane in $\R^3$ passing through
$f(p)$ perpendicular to the vector $d\hat\gamma(0)/dt$.
Then the intersection of the image of the singular set of $f$ by
$\Pi$ gives a $3/2$-cusp in the plane $\Pi$.
The value $\kappa_c(p)$ coincides with the cuspidal curvature
of this $3/2$-cusp (cf. \cite{MSUY}).
The following assertion holds:

\begin{fact}[\cite{MSUY}]\label{fact:prod}
 The value $|\kappa_c\kappa_\nu|$ is an intrinsic invariant.
\end{fact}

\begin{figure}[htbp]
 \centering
 \includegraphics[width=3.9cm]{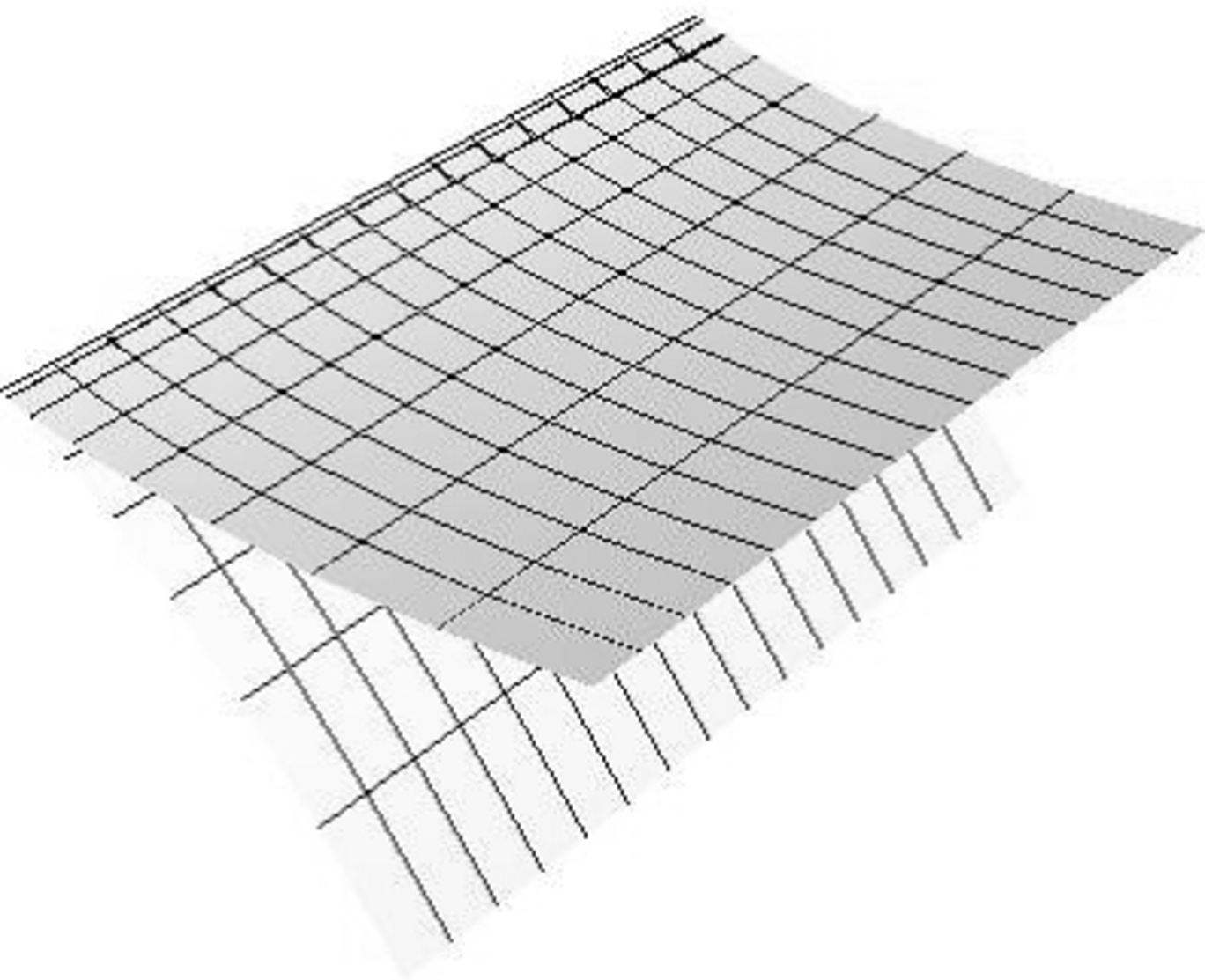}
 \qquad \quad 
 \includegraphics[width=3.9cm]{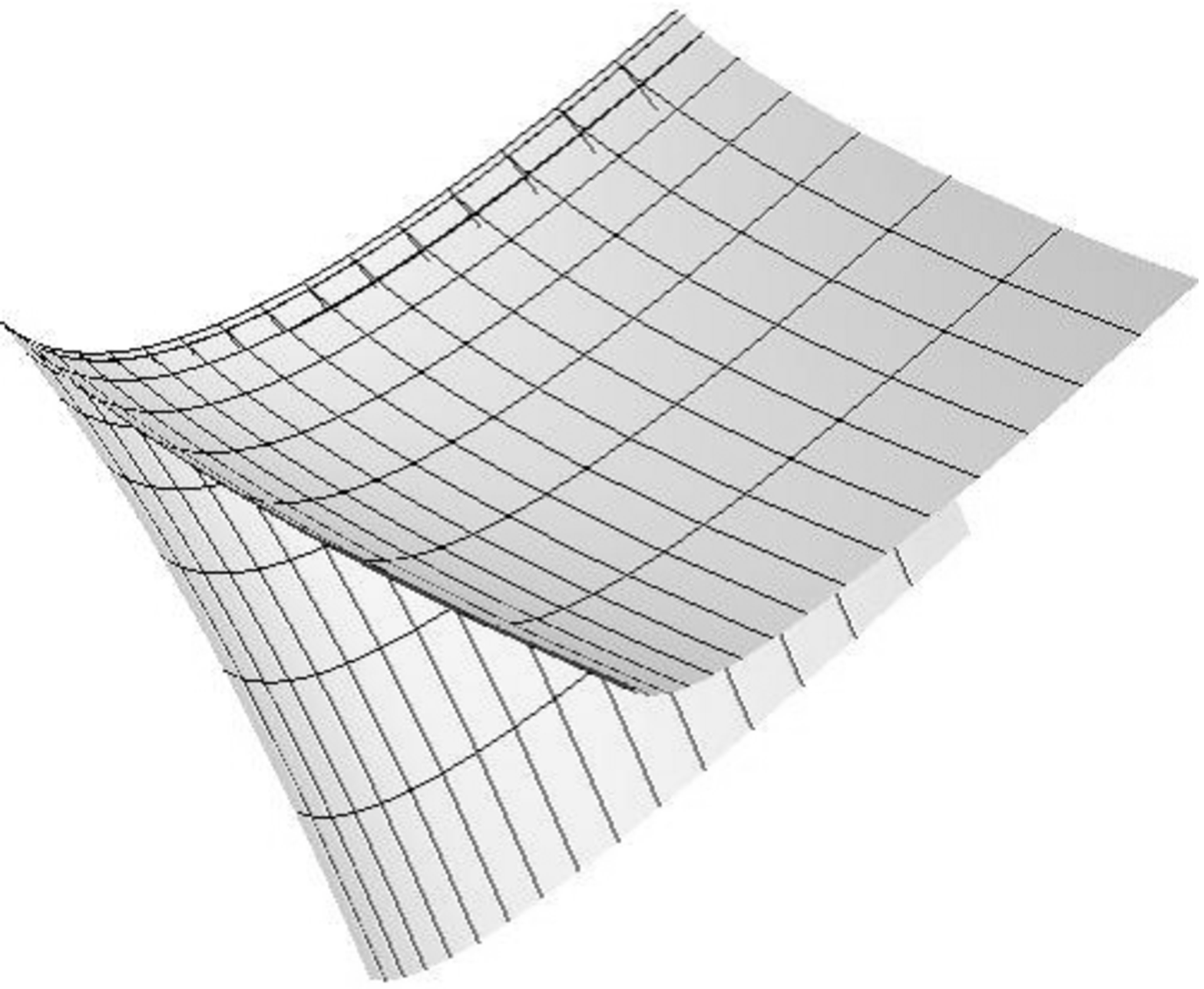} 
 \caption{%
    Cuspidal edges with $\kappa_\nu=0$ (left) 
    and $\kappa_\nu\ne 0$ (right)}
 \label{fig:generic}
\end{figure}

To prove our main result (cf.\ Theorem \ref{thm:A}),
this fact plays a crucial role.
A cuspidal edge singular point $p$ is
called \textit{generic} if the osculating plane
of $\hat \gamma$ is not orthogonal to $\nu$ at $p$,
that is, the limiting tangent plane does not
coincide with the osculating plane of $\hat \gamma$. 

The Gaussian curvature at a generic cuspidal edge
is unbounded (cf.\ \cite{MSUY}). 
Moreover, as shown in
\cite[Theorem A]{MSUY},
the following four conditions are equivalent:
\begin{itemize}
 \item[(a)] A cuspidal edge singular point
       is generic  (cf.\ Figure~\ref{fig:generic}).
 \item[(b)] The limiting normal curvature $\kappa_\nu$ does not vanish 
       at the singular point.
 \item[(c)] The inequality $\kappa>|\kappa_s|$
       holds (cf.\ \eqref{eq:kappa}).
 \item[(d)] Let $K$ be the Gaussian curvature
       and $d\hat A=\det(f_u,f_v,\nu)\,du\wedge dv$
       the signed area element of $f$, where $(U; u,v)$
       is a local coordinate system near
       the singular point $p$.
       Then $K\,d\hat A$ is well-defined on $U$ and
       does not vanish at $p$.
\end{itemize}
We denote by $\Cusp$ 
the set of \textit{real analytic} map germs of cuspidal edges 
$$
f\colon{}(\R^2,0)\to (\R^3,0)
$$
which is defined on a neighborhood of the origin in $\R^2$.
Moreover, let $\Cusp^*(\subset \Cusp)$ be the
set of map germs of generic cuspidal edges. 
In this paper, we show the following:
\begin{introthm}\label{thm:A}
 Let $\kappa_s(t)$ be the singular curvature
 function of the singular curve $\gamma(t)$
 of $f\in \Cusp^*$ such that $\gamma(0)=0$.
 Let $\sigma(t)$ be a real analytic regular space curve
 whose curvature function $\tilde\kappa(t)$
 satisfies 
 \begin{equation}\label{eq:A}
  \tilde\kappa(t)>|\kappa_s(t)|
 \end{equation}
 for all sufficiently small $t$.
 Then there exist at most two map germs
 $g\in \Cusp^*$ such that
 \begin{enumerate}
  \item the first fundamental form of $g$ coincides 
	with that of $f$, in particular, 
	the singular curve $\gamma(t)$ in the domain of $f$ is the same as that of $g$,
  \item $g(\gamma(t))=\sigma(t)$ holds for each $t$.
 \end{enumerate}
\end{introthm}

It is classically known that for a given
planar curve $\gamma(t)$ 
having curvature function $\kappa(t)$, 
there are at most two develpable surfaces having 
``origami-singularities'' as a space curve
whose curvature function $\tilde \kappa(t)$
satisfies $\tilde \kappa(t)>\kappa(t)$
(see \cite{FT}).
The above theorem can be considered 
as an analogue of this classical phenomenon. 

Kossowski \cite{K} is 
the first geometer who considered the realizing problem of 
given first fundamental forms as generic wave fronts.
However, in \cite{K}, the isometric deformations of 
singularities were not discussed, 
and the above theorem can be considered as 
a refinement of \cite[Theorem 1]{K} in the case
of cuspidal edge singularities.
Since the intrinsic formulation of wave fronts
are rather technical, the statement of Theorem \ref{thm:A}
as a refinement of Kossowski's realization theorem 
will be given and proved in the final section (Section~\ref{sec:realize}).
Since Kossowski applied the Cauchy-Kovalevskaya theorem to
construct suitable second fundamental forms,
our approach is completely different from his, and
can be beneficial for the applications to isometric deformations.
However, it should be also remarked that Kossowski's appoach will work
for not only cuspidal edges but also other wave front singularities,
(for example, swallowtail singular points).

We get the following consequences of  Theorem \ref{thm:A}:

\begin{introcor}\label{cor:B}
 Each map germ $f\in \Cusp^*$
 admits an isometric deformation $($in $\Cusp^*)$
 which moves the limiting normal curvature $\kappa_\nu$.
 In particular, $\kappa_\nu$ and $\kappa_c$
 are not intrinsic invariants%
\footnote{
 In \cite{MSUY},
 it has been shown that $\kappa_c$ is an extrinsic
 invariant, 
 by construction the isometric deformation of 
ruled cuspidal edges satisfying  $\kappa_\nu=0$.
 }.
\end{introcor} 
A given map germ $f\in \Cusp$
is said to be {\it planar\/} (resp.\ {\it non-planar})
if the image $\hat \gamma$ of the singular curve of 
$f$ is contained in a plane
(resp.\ the torsion of $\hat \gamma$ does not equal to zero).
We show the following normalization
theorem of generic cuspidal edges: 
\begin{introcor}\label{cor:C}
 For each $f\in \Cusp^*$,
 there exists a unique map germ 
 $g\in \Cusp^*$
 of planar cuspidal edge singularities up to congruence
 such that
 \begin{itemize}
  \item $f$ and $g$ induce the same first fundamental form, and
  \item the curvature function of 
	$\hat \gamma(t):=f\circ \gamma(t)$ coincides
	with that of $g\circ \gamma(t)$, where $\gamma(t)$ 
	is the singular curve of $f$.
 \end{itemize}
 Moreover, there exists a real analytic 
 isometric deformation of $f$
 into $g$ preserving the 
 curvature function along the image of the
 cuspidal edge singularities. 
\end{introcor}
In the statement of Corollary \ref{cor:C}, 
the condition $f\in\Cusp^*$ cannot be weakened to
$f\in\Cusp$ (cf.\ Remark \ref{rem:non-generic} and
Proposition \ref{prop:non-generic}). 
As a consequence, we also cannot weaken the condition
\eqref{eq:A} to $\tilde\kappa(t)\geq |\kappa_s(t)|$
in the statement of Theorem~\ref{thm:A}.

Two map germs $f,g\in \Cusp$ are said to be  {\it congruent\/}
if there exist an (orientation preserving or reversing)
isometry $T: (\R^3,0)\to (\R^3,0)$ and a local 
analytic diffeomorphism 
$\phi: (\R^2,0)\to (\R^2,0)$ such that $T\circ f\circ \phi=g$.
On the other hand,
two map germs $f,g\in \Cusp$ are said to be  
{\it strongly isometric\/} if there exist an isometry 
$T: (\R^3,0)\to (\R^3,0)$ and a local analytic diffeomorphism 
$\phi: (\R^2,0)\to (\R^2,0)$ satisfying the following
properties: 
\begin{itemize}
 \item $f\circ \phi$ and $g$ induce the same first fundamental form, and
 \item the restriction of $T\circ f\circ \phi$ to 
       the singular curve of $f$ coincides with that of $g$. 
\end{itemize}

\medskip
A regular space curve $\sigma$ passing through a point $x_0\in \R^3$
is called {\it symmetric\/}
if there exists an isometry $T$ of $\R^3$ 
which is not the identity map
such that $T({x_0})={x_0}$ and the image of
$\sigma$ is invariant under the action of $T$.
For example, the image of the singular curve
of generic planar cuspidal edges
are symmetric.
We get the following
duality theorem for generic cuspidal edges:
\begin{introcor}\label{cor:D}
 There exists an involution
 $\Cusp^*\ni f \mapsto \check f \in \Cusp^*$
 such that
 \begin{enumerate}
  \item $\check f$ is strongly isometric to $f$,
  \item if $g\in \Cusp^*$ is strongly isometric to $f$,
	then $g$ is congruent to $f$ or $\check f$.
 \end{enumerate}
 Moreover, $f$ is not congruent to $\check f$ if
 the image of the singular set of $f$ is non-symmetric 
 and non-planar.
\end{introcor}

We call $\check f$ the {\it isomer\/} of $f$.
For non-generic cuspidal edges, the existence of
isomers cannot be expected in general
(see Proposition \ref{prop:non-generic2}).
The proofs of these results are accomplished by
an appropriate modification of the 
proof of the $2$-dimensional case of the Janet-Cartan theorem.
In \cite{HHNUY}, we defined `normal cross caps',
expecting a similar normalization theorem as in
Corollary~\ref{cor:C}.
However, it seems difficult to apply the same technique, 
because cross cap singularities are isolated.

\section{Preliminaries}
The fundamental tool to prove our results is the following 
fact (cf.\ \cite[pages 37--38]{Sp}):

\begin{fact}[Cauchy-Kovalevskaya theorem]
\label{fact:ck}
 Let
 \[
    x^i_v(u,v)=\Phi^i(u,v,x^1,\dots,x^k,x^1_u,\dots,x^k_u)
     \qquad (i=1,\dots,k)
 \]
 be a partial differential equation having $x^i:=x^i(u,v)$ $(i=1,2,\dots,k)$
 as unknown functions, where 
 $\Phi:=(\Phi^1,\dots,\Phi^k)$ is a real analytic map
 and
 \[
    x^i_u:=\frac{\partial x^i}{\partial u},\qquad
    x^i_v:=\frac{\partial x^i}{\partial v}\qquad
    (i=1,\dots,k).
 \]
 This equation has a unique real analytic solution
 $x=(x^1,\dots,x^k)$ with an initial condition
 \[
   x^i(u,0)=w^i(u)\qquad (i=1,\dots,k),
 \]
 where $w^i$ $(i=1,\dots,k)$
 are given real analytic functions.
\end{fact}

The classical Janet-Cartan theorem for 
the existence of local isometric embeddings of
real analytic 
Riemannian $2$-manifolds can be proved as an application 
of this fact (cf.\ Chapter 11 of \cite{Sp}).
Our main results are also proved by applying Fact \ref{fact:ck}  using
the following special coordinate system along cuspidal edges:

\begin{definition}\label{def:adapted}
 Let $f:\Sigma^2\to \R^3$ be a real analytic map
 and $p\in \Sigma^2$ be a cuspidal edge singular point
 of $f$. 
 (We can take a real analytic
 unit normal vector field $\nu$  
 defined on a neighborhood of
 $p$.)  
 A real analytic local coordinate system $(u,v)$ at $p$
 is called {\it adapted\/}  if it satisfies
 the following properties along the $u$-axis:
 \begin{enumerate}
  \item\label{item:adapted:1}
       $|f_u| = 1$,
  \item\label{item:adapted:2}
       $f_v=0$, in particular, the singular set 
       is contained in
       the $u$-axis,
  \item\label{item:adapted:3}
       $\{f_u, f_{vv},\nu\}$ is an
       orthonormal basis which is compatible with respect to
       the orientation of $\R^3$.
 \end{enumerate}
\end{definition}

The existence of an adapted coordinate system was
shown in \cite[Lemma 3.2]{SUY}.
Throughout this paper, we fix a 
real analytic map 
\[
   f:(U;u,v)\longrightarrow \R^3
\]
defined on a domain $U\subset\R^2$ of the $uv$-plane
which has a generic cuspidal edge singular point
at the origin $(0,0)$, and assume that $(u,v)$ is an adapted coordinate
system.
Since $(u,v)$ is an adapted coordinate system,
the {\it limiting normal curvature\/} $\kappa_\nu$ of $f$ is given by
(cf.\ Equation (3.11) in \cite{SUY}) 
\begin{equation}
\label{eq:knu}
  \kappa_\nu(u):=f_{uu}(u,0)\cdot \nu(u,0)
   =\det(f_{uu}(u,0),f_u(u,0), f_{vv}(u,0)),
\end{equation}
where the dot \lq$\cdot$\rq\ is the inner product
in $\R^3$.
Since $f_v=0$ along the $u$-axis, there exists a 
real analytic map germ $\phi$
such that
\begin{equation}\label{eq:phi}
 f_v(u,v)=v \phi(u,v),\qquad
 f_{vv}(u,0)=\phi(u,0)
\end{equation}
hold on a neighborhood of the $u$-axis.

On the other hand, let 
$g:(U;u,v)\to \R^3$ be another real analytic map
which has a generic cuspidal edge singular point
at $(0,0)$, and $(u,v)$ an adapted coordinate system.
Similarly, there exists a  real analytic map germ $\psi$ such that
\begin{equation}\label{eq:psi}
 g_v(u,v)=v \psi(u,v),\qquad
 g_{vv}(u,0)=\psi(u,0)
\end{equation}
hold on a neighborhood of the $u$-axis.
\begin{proposition}\label{prop:F}
 Let $f$ and $g$ be as above, 
 and $\varphi$ and $\psi$ as in \eqref{eq:phi} and \eqref{eq:psi}.
 Suppose that the first fundamental form of $g$ 
 coincides with that of $f$, 
 then there exists a real analytic map
 $\F:U\times\R^3\times \GL_3(\R)\to \M_3(\R)$
 such that
 \[
    (g_v,r_v,\psi_v)=\F(u,v;\,\psi_u,(\psi,g_u,r_u)),
 \]
 where $r(u,v):=g_u(u,v)$
 and 
 $\M_3(\R)$ $($resp.\ $\GL_3(\R))$ is the 
 set of $3\times 3$-matrices 
 $($resp.\ the set of regular $3\times 3$-matrices$)$.
\end{proposition}
\begin{proof}
 Since $f$ and $g$ have the same first fundamental
 form on the same local coordinate system, we have
 \begin{equation}\label{eq:i1}
  f_u\cdot f_u=g_u\cdot g_u,\quad
  f_u\cdot f_v=g_u\cdot g_v,\quad
  f_v\cdot f_v=g_v\cdot g_v,
 \end{equation}
 that reduce to
 \begin{equation}\label{eq:i2}
   f_u\cdot f_u=g_u\cdot g_u,\quad
   f_u\cdot\phi=g_u\cdot \psi,\quad
   \phi\cdot \phi=\psi\cdot \psi.
 \end{equation}
 We define $\F:=(\F^1,\F^2,\F^3)$ by
 \begin{equation}\label{eq:F}
 \begin{aligned}
  &\F^1(u,v;\, {x},({y}_1,{y}_2,{y}_3)):= 
      v {y}_1,\\
  &\F^2(u,v;\, {x},({y}_1,{y}_2,{y}_3)):= v {x},\\
  &\F^3(u,v;\,x,(y_1,y_2,y_3))\\
  &\hphantom{\F^3(u,v;\,)}
   :=
   ((y_1,y_2,y_3)^T)^{-1}
   \pmt{\phi_v\cdot \phi\\ 
          \phi_v \cdot f_u\\[6pt]
          (\phi\cdot f_{uu})_v
          -\dfrac{v}2{(\phi\cdot \phi)_{uu}}+v (x\cdot x)
   },
 \end{aligned}
 \end{equation}
where $x\in \R^3$, 
 $(y_1,y_2,y_3)\in \GL_3(\R)$
 and $(y_1,y_2,y_3)^T$ is the transpose of 
 a regular matrix $(y_1,y_2,y_3)$.
 Since 
 \[
    g_v=v \psi,\qquad
    r_v=g_{uv}=(v \psi)_u=v \psi_u,
 \]
 it holds that
 \[
    \F^1(u,v;\, \psi_u,(\psi,g_u,r_u))=v \psi=g_v,\quad
    \F^2(u,v;\, \psi_u,(\psi,g_u,r_u))=v \psi_u=r_v.
 \]
 The relation (cf.\ \eqref{eq:phi} and \eqref{eq:psi})
 \[
    v(g_u\cdot \psi)=g_u\cdot g_v=f_u\cdot f_v=v(f_u\cdot \phi)
 \]
 reduces to
 \begin{equation}\label{eq:a1}
  g_u\cdot \psi=f_u\cdot \phi.
 \end{equation}
 Since (cf.\ \eqref{eq:phi} and \eqref{eq:i2})
 \begin{align*}
   v(\psi\cdot g_{uu})
     &=(g_v\cdot g_{uu})=(g_v\cdot g_u)_u-g_{uv}\cdot g_u \\
     &=(g_v\cdot g_u)_u-
         \frac{1}{2}(g_u\cdot g_u)_v
      =(f_v\cdot f_u)_u-
         \frac{1}{2}(f_u\cdot f_u)_v
      =v (\phi\cdot f_{uu}), 
 \end{align*}
 it holds that
 \begin{equation}\label{eq:*}
    g_{uu}\cdot \psi=f_{uu}\cdot\phi.
 \end{equation}
 On the other hand, by \eqref{eq:a1} and \eqref{eq:i2},
 we have that
 \begin{align}\label{eq:a2}
  \psi_v\cdot g_u
    &=(\psi\cdot g_u)_v-(\psi\cdot g_{uv})
     =(\phi\cdot f_u)_v-\frac1v g_v\cdot g_{uv}\\
    &=(\phi\cdot f_u)_v-\frac1{2v} (g_v\cdot  g_{v})_{u}
     =(\phi\cdot f_u)_v-\frac1{2v} (f_v\cdot f_{v})_u
      \nonumber \\
    &=(\phi\cdot  f_u)_v-\frac1{v} f_v\cdot  f_{uv}  
     =\phi_v\cdot f_u. \nonumber
 \end{align}
 Since
 \begin{align*}
    v(\psi\cdot g_{uuv})
    &=g_v\cdot g_{uuv}=(g_v\cdot g_{uv})_u-g_{uv}\cdot g_{uv}
     =\frac{(g_v\cdot g_{v})_{uu}}2-g_{uv}\cdot g_{uv} \\
    &=\frac{(f_v\cdot f_{v})_{uu}}2-g_{uv}\cdot g_{uv}
     =v^2\frac{(\phi\cdot \phi)_{uu}}2-v^2 \psi_{u}\cdot \psi_{u},
 \end{align*}
 we have 
 \begin{equation}\label{eq:b1}
  \psi\cdot g_{uuv}=\frac12v(\phi\cdot \phi)_{uu}-v(\psi_{u}\cdot \psi_{u}).
 \end{equation}
 Here, \eqref{eq:*} yields
 \begin{equation}\label{eq:b2}
  \psi\cdot g_{uuv}=(\psi\cdot g_{uu})_v-\psi_v\cdot g_{uu}
                   =
                      (\psi\cdot g_{uu})_v-\psi_v\cdot r_u.
 \end{equation}
 By \eqref{eq:b1} and \eqref{eq:b2}, we have 
 \begin{equation}\label{eq:a3}
  \psi_v\cdot r_{u}
   = 
   (\phi\cdot f_{uu})_v
     -\frac{v}2{(\phi\cdot \phi)_{uu}}+v (\psi_u\cdot \psi_u).
 \end{equation}
 The space curve $\sigma(u):=g(u,0)$
 parametrizes the image of the singular set of $g$.
 Since the cuspidal edge singularities 
of $g$ are generic, the osculating plane $\Pi$
 of the space curve $\sigma(u)$ is independent of the tangential
 direction $\psi(u,0)=g_{vv}(u,0)$ of $g$ 
 (cf.\ \eqref{eq:psi}, Definition~\ref{def:adapted}).
 Since $\Pi$ is spanned by $\{g_{u}(u,0),g_{uu}(u,0)\}$,
 the matrix 
 \[
   (\psi(u,0),g_u(u,0),r_u(u,0))
       =(g_{vv}(u,0),g_{u}(u,0),g_{uu}(u,0))
 \]
 is regular. 
 Hence,
 \eqref{eq:a1}, \eqref{eq:a2} and \eqref{eq:a3}
 yield that
 \[
    \F^3(u,v;\, \psi_u,(\psi,g_u,r_u))=\psi_v,
 \]
 in particular, $\F=(\F^1,\F^2,\F^3)$
 attains the desired real analytic map.
\end{proof}

\section{Proof of the main results}
Let $f\colon(U;u,v)\to\R^3$ be as in the previous section.
Then the space curve defined by
\[
    \hat\gamma(t):=f(t,0)
\]
gives a parametrization of the image of the singular curve
$\gamma(t)=(t,0)$ of $f$.
To prove Theorem~\ref{thm:A} in the introduction, 
we prepare the following assertion:

\begin{proposition}\label{prop:ini}
 Let $\sigma(t)$ be a regular space curve satisfying \eqref{eq:A}
 such that $t$ is the arclength parameter.
 Then  there exists a unique $\R^3$-valued vector field $X^{+}_{\sigma}(t)$ 
 $($resp. $X^{-}_{\sigma}(t))$
 along $\sigma$ satisfying the following properties:
 \begin{enumerate} 
  \item\label{item:w1} 
       $|X^{+}_{\sigma}(t)|=1$ 
       {\rm(}resp.\ $|X^{-}_{\sigma}(t)|=1${\rm)},
  \item\label{item:w2} 
       $X^{+}_{\sigma}(t)\cdot \dot \sigma(t)=0$ 
       {\rm(}resp.\ $X^{-}_{\sigma}(t)\cdot \dot \sigma(t)=0${\rm)},
  \item\label{item:w3} 
       $X^{+}_{\sigma}(t)\cdot \ddot \sigma(t)=\phi(t,0)\cdot 
           \ddot \gamma(t)$
       {\rm(}resp.\ $X^{-}_{\sigma}(t)\cdot \ddot \sigma(t)
                    =\phi(t,0)\cdot \ddot \gamma(t)${\rm)}, 
          where $\phi$ is given in \eqref{eq:phi},
  \item\label{item:w4} 
       $\det(\dot \sigma(t), X^{+}_{\sigma}(t), \ddot \sigma(t))>0$
       {\rm (}resp.\ 
       $\det(\dot \sigma(t), X^{-}_{\sigma}(t), \ddot \sigma(t))<0${\rm)}.
 \end{enumerate}
\end{proposition}

\begin{proof}
 Since $f$ has non-zero limiting normal curvature,
 the curvature function $\kappa(t)$ of the space curve
 $\hat\gamma(t)=f(t,0)$ is positive  (cf.\ \eqref{eq:kappa}).
 Since $\phi(t,0)=f_{vv}(t,0)$ is a unit vector field 
 (cf.\ Definition~\ref{def:adapted})
 orthogonal to $\nu(t,0)$ and $\dot {\hat \gamma}(t)$,
 we have that
 \[
\left|\phi(t,0)\cdot\ddot{\hat \gamma}(t)\right|=\left|\kappa_s(t)\right|.
 \]
 Let $\tilde\kappa(t)$ be the curvature function of $\sigma(t)$. 
 If we set 
 \begin{equation}\label{eq:ct}
    c(t):=\phi(t,0)\cdot \frac{\ddot{\hat \gamma}(t)}{\tilde\kappa(t)},
 \end{equation}
 the inequality \eqref{eq:A} yields
 \[
    |c(t)|=\frac{|\kappa_s(t)|}{\tilde \kappa(t)}< 1.
 \]
 On the other hand, since $t$ is the arclength parameter of
 $\sigma$, 
 \[
    n(t):=\frac{\ddot\sigma(t)}{|\ddot\sigma(t)|}=
            \frac{\ddot\sigma(t)}{\tilde\kappa(t)}
 \]
 gives the principal unit normal vector
 field of the space curve $\sigma(t)$.
 Then applying the lemma in the appendix for 
 $a=n(t)$, $b=\dot\sigma(t)$
 and $\mu=c(t)$ as in \eqref{eq:ct},
 we can verify that $X^{+}_\sigma(t):=w$ 
 (resp.\ $X^{-}_\sigma(t):=w$) satisfies
 \ref{item:w1}--\ref{item:w4}.
 The uniqueness of $X^{+}_\sigma(t)$ and
 $X^{-}_\sigma(t)$ also follows from the lemma in the appendix.
\end{proof}
\begin{proof}[Proof of Theorem~\ref{thm:A}]
 We consider the 
 partial differential equation
 (cf.\ Proposition \ref{prop:F})
 \begin{equation}\label{eq:F2}
  (g_v,r_v,\psi_v)
   =\F(u,v;\,\psi_u,(\psi,g_u,r_u))
 \end{equation}
 for $\F$ as in \eqref{eq:F}
 with the following initial conditions:
 \begin{equation}\label{eq:F20}
  \begin{alignedat}{2}
   g(u,0)&=\sigma(u),\qquad
   &&r(u,0)=\dot\sigma(u),\\
   \psi(u,0)&=X^+_\sigma(t)  \qquad&&(\mbox{resp. }\,\,
   \psi(u,0)=X^-_\sigma(t) ).
  \end{alignedat}
 \end{equation}
 By Fact~\ref{fact:ck}, there exists a
 real analytic map
 \[
   g^{+}:U\to \R^3 \quad(\mbox{resp. }\,\, g^{-}:U\to \R^3)
 \]
 on a sufficiently small neighborhood $U$ of 
 the origin satisfying \eqref{eq:F2} and
 \eqref{eq:F20}.
 Since \eqref{eq:F2} is equivalent to the conditions
 \begin{align}
  & g^{+}_v=v \psi,\quad  (\mbox{resp. } g^{-}_v=v \psi),
  \label{e:1}\\ 
  & r_v=v \psi_u, \label{e:2}\\
  & \psi_v\cdot \psi=\phi_v\cdot \phi,\label{e:4} \\ 
  & \psi_v\cdot g^{+}_u
  =\phi_v\cdot f_u \quad 
  (\mbox{resp. } \psi_v\cdot g^{-}_u=\phi_v\cdot f_u), \label{e:5}\\
  & \psi_v\cdot r_{u}
  =(\phi\cdot f_{uu})_v
  -\frac{v}2{(\phi\cdot \phi)_{uu}}+v (\psi_u\cdot \psi_u)
  \label{e:6},
 \end{align}
 one can deduce the relation \eqref{eq:i2} (and also
 \eqref{eq:i1} as a consequence) directly from the 
 relations \eqref{eq:F20} and
 \eqref{e:1}--\eqref{e:6}.
 Consequently, the map $g^{+}$ (resp. $g^-$)
 has the same first fundamental
 form as $f$ and satisfies the properties (2) and (3)  of Theorem~\ref{thm:A}.
 The ambiguity of the construction of $g$ as
 in the statement of Theorem~\ref{thm:A} depends on the choice of
 the vector field $X$ along $\sigma$
 satisfying \ref{item:w1},
 \ref{item:w2} 
 and \ref{item:w3}.
 By the lemma in  the appendix,
 $X$ coincides with either $X^+_\sigma$
 or $X^-_\sigma$, which yields 
 at most two possibilities for $g$.
 The proof of Theorem~\ref{thm:A} is now reduced
 to the following Proposition \ref{prop:ce}.
\end{proof}

\begin{proposition}\label{prop:ce}
 The real analytic map $(g:=)g^{\pm}$ has 
 generic cuspidal edge singularities along the $u$-axis.
\end{proposition}

\begin{proof}
 Since $(u,v)$ is an adapted coordinate system of $g$,
 the vector field  $\psi$ in \eqref{eq:psi} is
perpendicular to
 $g_u$
 on the singular set.
 Moreover, by the definition \eqref{eq:psi} of $\psi$,
 \[
    \tilde\nu(u,v):=   \frac{g_u(u,v)\times \psi(u,v)}{|g_u(u,v)\times \psi(u,v)|}
 \]
 is a unit normal vector field to $g$ which is well-defined on the
 singular set,  where $\times$ denotes the vector product in $\R^3$.
 Moreover, the function
 \[
    \lambda:=\det(g_u,g_v,\tilde\nu)= 
     v\det(g_u,\psi,\tilde \nu)
 \]
 satisfies $\lambda_v\neq 0$ on the $u$-axis,
 because $(u,v)$ is an adapted coordinate system.
 Hence the singular points are {\it non-degenerate\/} 
 (cf.\ \cite[Proposition 2.3]{KRSUY} or \cite[Definition 1.1]{SUY}).
 Moreover, since $g_v=v\psi=0$ on the $u$-axis,
 the {\em null direction\/} 
at a point on the $u$-axis
 (cf.\ \cite[Page 306]{KRSUY} or \cite[Page 495]{SUY})
 is $\partial/\partial v$, which is linearly independent
 to the {\em singular direction\/} $\partial/\partial u$.
 Thus, to show that $g$ is a cuspidal edge, 
 it is sufficient to show that $g$ 
 is a front (cf.\ \cite{KRSUY} or \cite{SUY}), 
 which is equivalent
 to $\tilde\nu_v\neq 0$ on the $u$-axis:
 Let $\kappa_c$ and $\kappa_\nu$
 be the cuspidal curvature (cf.\ \cite[(2.4)]{MSUY})
 and the limiting normal curvature of $f$ along  $\gamma$,
 respectively.
 Since the singularity of $f$ consists of cuspidal edges, 
 $\kappa_c\neq 0$ holds (cf.\ \cite[Lemma 2.8]{MSUY}), 
 and since $f$ is generic (i.e. the singularities of
$f$ cossists only of generic cuspidal edges), 
$\kappa_\nu\neq 0$.
 By Fact \ref{fact:prod},
 $|\kappa_c\kappa_\nu|$ 
 depends only on the first fundamental from.
 Thus, we have
 \begin{equation}\label{eq:generic}
    \tilde \kappa_c(t) \tilde \kappa_\nu(t)=
    \kappa_c(t)\kappa_\nu(t)\ne 0,
 \end{equation}
 where $\tilde\kappa_c$ and $\tilde\kappa_\nu$ are 
 the cuspidal curvature and the limiting normal curvature of 
 $g$, respectively.
 Then by \cite[(2.4)]{MSUY},
 \begin{equation}\label{eq:kc}
    \tilde{\kappa}_c (t) = \det(g_u,g_{vv},g_{vvv})|_{(u,v)=(t,0)}
         = 2\det(g_u(t,0),\psi(t,0),\psi_v(t,0))\neq 0.
 \end{equation}
 So it holds on the $u$-axis that
 \begin{align*}
  \tilde\nu_v\cdot \psi&=
  \left(\frac{g_u\times\psi}{|g_u\times\psi|}\right)_v \cdot \psi\\
  &=
  \left(\frac{g_{uv}\times\psi+g_u\times\psi_v}{|g_u\times\psi|}\right)
     \cdot\psi 
  + \bigl((g_u\times\psi)\cdot\psi\bigr)\left(\frac{1}{|g_u\times\psi|}\right)_v\\
  &=-\frac{\det(g_u,\psi,\psi_v)}{|g_u\times\psi|}\neq 0.
 \end{align*}
 Hence the singular points of $g$ consist of cuspidal edge
 singularities.
 Moreover, by \eqref{eq:kc}, the limiting normal curvature
 $\tilde\kappa_\nu$
 does not vanish, which implies that $g$ is generic.
\end{proof}

\begin{proof}[Proof of Corollary \ref{cor:B}]
 Since $f$ has non-zero limiting normal curvature,
 the curvature function $\kappa(t)$ of the space curve
 $\hat\gamma(t)=f(t,0)$ is 
 greater than the absolute value $|\kappa_s(t)|$ of the
 singular curvature (cf. \eqref{eq:kappa}).
 Let $\tau(t)$ be 
 the torsion function of $\hat \gamma(t)$.
 For sufficiently small $\varepsilon>0$,
 there exists a regular space curve $\sigma^s(t)$
 ($|s|<\varepsilon$)  satisfying the following properties
 by the fundamental theorem of space curves
 \begin{itemize}
  \item $\sigma^s(0)=0$,
  \item $\sigma^0(t)=\hat \gamma(t)$,
  \item the curvature function of $\sigma^s(t)$ is equal to $\kappa(t)+s$,
  \item the torsion function of $\sigma^s(t)$ is equal to $\tau(t)$.
 \end{itemize}
 Since $\varepsilon$ is sufficiently small,
 we may assume that $\kappa(t)+s>|\kappa_s(t)|$.
 By Theorem~\ref{thm:A}, there exists $g^s\in \Cusp^*$
 ($|s|<\varepsilon$)
 such that
 \begin{enumerate}
  \item the vector field $\psi^s(u,v)$ satisfying $g^s_v=v \psi^s$
	is equal to $X^+_{\sigma^s}(u)$ along $v=0$,
	where $X^+_{\sigma^s}$ is a vector field along $\sigma^s$
	defined in Proposition~\ref{prop:ini},
  \item the first fundamental form of $g^s$ is equal to that of $f$,
  \item the singular curve $\gamma(t)$ of $f$ is the same as that of $g^s$, and
  \item $g^s(\gamma(t))=\sigma^s(t)$ holds for each $t$.
 \end{enumerate}
 Since the geodesic curvature $\kappa_s$ is intrinsic,
 \eqref{eq:kappa} yields that
 $\sqrt{(\kappa(t)+s)^2-\kappa_s(t)^2}$
 is equal to the absolute value of the limiting normal curvature of
 $g^s$, 
 which proves Corollary \ref{cor:B}.
\end{proof}
\begin{proof}[Proof of Corollary \ref{cor:C}]
 Let $\kappa(t)$ and $\tau(t)$ be the curvature function and the torsion
 function of the space curve $\hat\gamma(t)$, respectively.
 For each $s\in [0,1]$,
 there exists a regular space curve $\sigma^s(t)$
 satisfying the following properties
 by the fundamental theorem of space curves
 \begin{itemize}
  \item $\sigma^s(0)=0$,
  \item $\sigma^0(t)=\hat \gamma(t)$,
  \item the curvature function of $\sigma^s(t)$ is equal to $\kappa(t)$,
  \item the torsion function of $\sigma^s(t)$ is equal to $(1-s)\tau(t)$.
 \end{itemize}
 Since $f$ is generic,
 $\kappa(t)>|\kappa_s(t)|$ holds.
 Then by Theorem~\ref{thm:A}, there exists $g^{s,+}\in \Cusp^*$
 (resp.\ $g^{s,-}\in \Cusp^*$) 
 for $s\in [0,1]$
 such that
 \begin{enumerate}
  \item the vector field $\psi^{s,+}(u,v)$ 
	(resp.\ $\psi^{s,-}(u,v)$) satisfying 
	$g^{s,+}_v=v \psi^{s,+}$ (resp.\ $g^{s,-}_v=v \psi^{s,-}$)
	is equal to $X^+_{\sigma^s}(u)$ (resp.\ $X^-_{\sigma^s}(u)$) 
	along $v=0$, 
	where $X^{\pm}_{\sigma^s}$ are as in Proposition~\ref{prop:ini}.
  \item the first fundamental form of $g^{s,+}$ (resp.\ $g^{s,-}$)
	is equal to that of $f$,
  \item the singular curve $\gamma(t)$ of $f$ is 
	the same as that of $g^{s,+}$ (resp. $g^{s,-}$), and
  \item $g^{s,+}(\gamma(t))=\sigma^s(t)$ 
	(resp.\ $g^{s,-}(\gamma(t))=\sigma^s(t)$)
	holds for each $t$.
 \end{enumerate}
 By this construction,  the torsion function of
 the curve  $\sigma^1$ vanishes identically.
 So we can conclude that $g^{1,+}$ 
 (resp.\ $g^{1,-}$)
 is a germ of a planar cuspidal edge.
 In particular, $\sigma^1$
 lies in a plane $\Pi$, and the
 curve is invariant under the reflection with respect to the plane $\Pi$.
 Let $T$ be the reflection with respect to the plane $\Pi$.
 Then  we have
 \[ 
    T\circ\sigma^{1}=\sigma^{1},\quad
    dT(\dot\sigma^{1})=\dot\sigma^{1},\quad
    dT(\ddot\sigma^{1})=\ddot\sigma^{1},\quad
    dT(X^{+}_{\sigma^1})=X^-_{\sigma^1}.
 \]
 In fact, the lemma in the appendix implies the
 fourth equality.
 Thus, by the uniqueness of Fact~\ref{fact:ck}, 
 we have $g^{1,-}=T\circ g^{1,+}$.
 This implies the uniqueness of $g$ as in Corollary~\ref{cor:C}. 
\end{proof}

\begin{proof}[Proof of Corollary \ref{cor:D}]
 By replacing $\nu$ with $-\nu$,
 we may assume that the limiting normal curvature $\kappa_\nu$
 of $f$ takes positive values without loss of generality.
 By Theorem~\ref{thm:A}, there exists $g^{+}\in \Cusp^*$
 {\rm (}resp.\ $g^{-}\in \Cusp^*${\rm)}
 for $s\in [0,1]$ such that
 \begin{enumerate}
  \item   the vector field $\psi^{+}(u,v)$ (resp.\ $\psi^{-}(u,v)$)
	  satisfying $g^{+}_v=v \psi^{+}$ (resp.\ $g^{-}_v=v \psi^{-}$)
	  is equal to $X^+_{\hat\gamma}(u)$ (resp.\
	  $X^-_{\hat\gamma}(u)$) 
	  along $v=0$, 
	  where $X^{\pm}_{\hat\gamma}$ are as in
	  Proposition~\ref{prop:ini},
  \item the first fundamental form of $g^{+}$ (resp.\ $g^{-}$)
	is equal to that of $f$,
  \item the singular curve $\gamma(t)$ of $f$ is the same 
	as that of $g^{+}$ (resp. $g^{-}$), and
  \item $g^{+}(\gamma(t))=\hat\gamma(t)$ 
	(resp.\ $g^{-}(\gamma(t))=\hat\gamma(t)$)
	holds for each $t$.
 \end{enumerate}
 Since $(u,v)$ is an adapted coordinate system of $f$, it holds that
 \begin{equation}\label{eq:c1}
    |\phi(t,0)|=|f_{vv}(t,0)|=1,\qquad 
     \phi(t,0)\cdot \dot {\hat \gamma}(t)
       =f_{vv}(t,0)\cdot f_u(t,0)=0.
 \end{equation}
 Since $\ddot {\hat \gamma}\cdot \nu =\kappa_\nu>0$ 
 and $\nu(t,0)=f_u(t,0)\times f_{vv}(t,0)$,
 we have 
 \begin{equation}\label{eq:c2}
  \det(\dot{\hat \gamma}(t), \phi(t,0), \ddot {\hat \gamma}(t))>0.
 \end{equation}
 So the uniqueness of $X^+_{\hat \gamma}$ implies that
 $X^+_{\hat \gamma}(t)=\phi(t,0)$ holds.
 Thus we have that
 \[
    g^+(u,v)=f(u,v).
 \]
 We now define an involution
 \[
    \Cusp^* \ni f\longmapsto \check f:=g^-\in \Cusp^*.
 \]
 It can be easily checked that
 $\check f(u,v):=g^-(u,v)$ is strongly isometric to $f(u,v)(=g^{+}(u,v))$.
 From now on, we suppose that the image of the singular curve
 of $f$ is non-symmetric and non-planar.
 To prove Corollary~\ref{cor:D},
 it is sufficient to
 show that $\check f$ 
 is not congruent to $f$.
 Suppose that there exists an isometry $T$ in $\R^3$ such that 
 \begin{equation}\label{eq:T}
     f(u,v)=T\circ \check f(\xi(u,v),\eta(u,v)),
 \end{equation}
 where $(u,v)\mapsto(\xi(u,v),\eta(u,v))$
 is a local analytic diffeomorphism 
 such that
 \[
     \bigl(\xi(0,0),\eta(0,0)\bigr)=(0,0).
 \]
\begin{lemma}\label{lem:xu}
  Under the situation above, we have
  \[
    \xi(u,0)=\epsilon u,\qquad\xi_v(u,0)=0,\qquad \eta(u,0)=0,
 \]
 where $\epsilon=\pm 1$.
\end{lemma}
\begin{proof}
 Since $T$ is an isometry and $(u,v)$ (resp.\ $(\xi,\eta)$) is 
 an adapted coordinate system for $f$ (resp.\ $\check f$),
 the singular set $\{v=0\}$ of $f$ coincides with the 
 singular set $\{\eta=0\}$ of $\check f$.
 Hence we have
 \[
     \eta(u,0)=0.
 \]
 Let $\tilde f(u,v)=\check f\bigl(\xi(u,v),\eta(u,v)\bigr)$.
 Since $f=T\circ\tilde f$, 
 \begin{align*}
    1 &= f_u(u,0)\cdot f_u(u,0)
      = \tilde f_u(u,0)\cdot \tilde f_u(u,0)\\
      &= \left|\xi_u(u,0)\check f_{\xi}\bigl(\xi(u,0),\eta(u,0)\bigr)
            +\eta_u(u,0)\check f_{\eta}\bigl(\xi(u,0),\eta(u,0)\bigr)
        \right|^2\\
      &       =|\xi_u(u,0)|^2.
 \end{align*}
 Here, we used the fact that $(\xi,\eta)$ is an adapted coordinate
 system for $\check f$.
 Hence we have $\xi_u(u,0)=\epsilon$ ($\epsilon=\pm 1$).
 Since $\xi(0,0)=0$, we have the first conclusion.

 On the other hand, $\partial/\partial v$ 
 (resp.\  $\partial/\partial\eta$) is the null direction of
 $f$ (resp.\ $\check f$) along the singular curve,
 so it holds that
 \begin{align*}
   0 &= f_v(u,0)  
     = \tilde f_v(u,0)\\
     &= \xi_v(u,0)f_\xi\bigl(\xi(u,0),\eta(u,0)\bigr)
       +
       \eta_v(u,0)f_\eta\bigl(\xi(u,0),\eta(u,0)\bigr)\\
     & =\xi_v(u,0)f_\xi\bigl(\xi(u,0),0\bigr).
 \end{align*}
 Since $|f_{\xi}|=1$ on the singular set, 
 we have
 $\xi_v(u,0)=0$.
\end{proof}

Since $\hat \gamma(t)$
is non-planar,
its torsion function
does not vanish.
Recall that the torsion function of a regular space curve does not depend on
the choice of orientation of the curve,
but changes sign by orientation reversing isometries of $\R^3$.
Hence the isometry $T$ as in \eqref{eq:T}
must be orientation preserving.
Since
\[
  \tilde f_{uu}=\xi_{uu}\check f_\xi+\eta_{uu}\check f_\eta 
       +(\xi_u)^2\check f_{\xi\xi}
        +2\xi_u\eta_u\check f_{\xi\eta}
        +(\eta_u)^2\check f_{\eta\eta},
\]
where $\tilde f(u,v)=\check f(\xi(u,v),\eta(u,v))$,
Lemma~\ref{lem:xu} implies that
$\tilde f_{uu}(u,0)= \check f_{\xi\xi}(\xi(u,0),0)$.
Since $f=T\circ\tilde f$ and $\xi_u(u,0)^2=1$,  it holds that,
\begin{align*}
   0<\kappa_\nu(u)
    &=\det(f_u(u,0),f_{vv}(u,0),f_{uu}(u,0)) \\
    &=\det(T \circ \tilde f_u(u,0),T \circ \tilde f_{vv}(u,0),
         T \circ \tilde f_{uu}(u,0))\\
    &=\det(\tilde f_u(u,0),\tilde f_{vv}(u,0),\tilde f_{uu}(u,0))\\
    &=\xi_u(u,0)
       \det(\check f_\xi(\xi(u,0),0), \check f_{\eta\eta}(\xi(u,0),0),
           \check f_{\xi\xi}(\xi(u,0),0)).
\end{align*}
By definition of $\check f(=g^-)$, it holds that
\[
    \det(\check f_\xi(0,0), \check f_{\eta\eta}(0,0),\check f_{\xi\xi}(0,0))
    =
    \det(\dot{\hat \gamma}(0),X^-_{\hat \gamma}(0),
        \ddot{\hat \gamma}(0))<0.
\]
So we can conclude that $\xi_u(u,0)=-1$, and by Lemma~\ref{lem:xu},
it holds that
\begin{equation}\label{eq:xi}
   \xi(u,0)=-u.
\end{equation}
Then $t$ is a common arclength parameter of $\hat \gamma(t)$
and $\check \gamma(t):=\check f(\xi(t,0),0)$.
Hence we have 
\[
   \check \gamma(-u)
  =\check \gamma(\xi(u,0))=
   T\circ \check f(\xi(u,0),\eta(u,0))=f(u,0)=\hat \gamma(u),
\]
that is, the curve $\hat\gamma$ is 
symmetric at the origin ($=\hat\gamma(0)$), 
which contradicts our assumption.
Hence $\check f$ cannot be congruent to $f$.
By the definition of strongly isometric equivalence,
it is obvious that $\check f$ is strongly isometric to $f$,
and (2) of Corollary \ref{cor:D} holds.
\end{proof}

\begin{remark}\label{rem:non-generic}
 For a real analytic map germ
 $f$ of a non-generic cuspidal edge singularity,
 the partial differential equation \eqref{eq:F2}
 cannot be solved, since
 $f_u,f_{uu},\phi$ are not linearly independent.
 Cuspidal edges on surfaces of constant Gaussian curvature
 are all non-generic (cf.\ \cite{MSUY}).
 Moreover, as shown in Proposition~\ref{prop:non-generic} below,
 Theorem~\ref{thm:A} does not hold when $f$ is not generic. 
 Although our method is not effective for such 
 surfaces, examples of isometric deformations of
 flat cuspidal edges with vanishing $\kappa_\nu$
 are given in \cite{HHNUY} and \cite{MSUY}.
\end{remark}

\begin{proposition}\label{prop:non-generic}
Let $\gamma(t)$ be the singular curve of
the cuspidal edge singularities of a $C^\infty$-map
$f$ having vanishing Gaussian curvature
on its regular set. 
Suppose that $\hat\gamma(t):=f\circ \gamma(t)$
lies in a plane in $\R^3$.
Then the image of the curve $\hat \gamma$ lies 
in a straight line\footnote{Let $C$ be a $3/2$-cusp
on $xy$-plane in $\R^3$. By considering
a cylinder or a cone over $C$, one can actually
get a flat cuspidal edge whose image of
singular set is contained in a line.
}.
\end{proposition}

\begin{proof}
 The map $f$ is a flat front in the sense of \cite{MU}%
\footnote{
 A front whose Gaussan caurvature vanishes 
 is called a flat front. 
 The precise definition is given in  \cite{MU}.
}.
 By \cite[Proposition 1.10]{MU},
 the singular points of $f$ are not umbilical points.
 Then by \cite[Proposition 2.2]{MU},
 $f$ is developable. 
 Since $\kappa_\nu$ vanishes along the
 singular curve $\gamma$, the asymptotic
 direction at $\hat\gamma(t)$ is $\dot{\hat\gamma}(t)$.
 In particular, $f$ is a tangential developable surface, 
 that is, we may set
 $f(u,v)=\hat\gamma(u)+v \dot{\hat\gamma}(u)$.
 Then the unit normal vector field $\nu$ of $f$
 is equal to the binormal vector of $\hat \gamma(u)$.
 Suppose that $\hat\gamma(t)$ is a regular curve
 with non-zero curvature function which lies in
 a plane.
 Since $f$ is a front, the fact $\nu_v=0$ implies that
 $\nu_u$ does not vanish, that is, the torsion function of
 $\hat \gamma$ does not vanish, which contradicts
 the fact that $\hat\gamma(t)$ is a planar.
 Thus, the curvature function of $\hat\gamma(t)$ vanishes
 identically, namely, its image lies in a straight line.
\end{proof}

In Corollary~\ref{cor:D}, we have shown the existence
of isomers of given cuspidal edges.
However, for the case of developable surfaces
(they are non-generic),
there are no such isomers:

\begin{proposition}\label{prop:non-generic2}
 Let $\sigma(t)$ be a regular space curve
 whose curvature function $\kappa(t)$
 and torsion function $\tau(t)$ have
 no zeros. 
 Then there exists a unique 
 flat front
 germ which has cuspidal 
 edge singularities along $\sigma$.
\end{proposition}

\begin{proof}
 Let $f$ be a 
 flat front which has cuspidal 
 edge singularities along $\sigma$.
 By \cite[Proposition 1.10]{MU},
 the singular set of $f$ cannot be
 umbilical points.
 So \cite[Proposition 2.2]{MU} implies that
 $f$ is a developable surface.
 In paticular, $f$ must be a tangential
 developable of $\sigma$, which proves the assertion.
\end{proof}

\begin{figure}[htb]
 \centering
 \includegraphics[width=0.6\hsize]{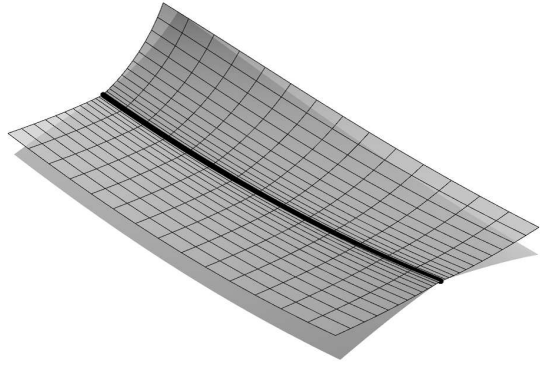}
 \caption{%
    The images of $f(u,v)$ (bottom-left) 
    and $g(u,v)$ (top-right) respectively,
    which have the common image of the singular set 
     (the range of the map is 
    $|u|<1/8$ and $|v|<1/4$).
}
 \label{fig:isomer_ex}
\end{figure}

Finally, we shall give a concrete example:
We set
\[
   f(u,v) := 
     \left( u,
          \ -\frac{v^2}{2}+\frac{u^3}{6},\ 
          \frac{u^2}{2}+\frac{u^3}{6}+\frac{v^3}{6}
     \right),
\]
whose singular set consists of generic cuspidal edges, and
is parametrized by
\[
   \hat\gamma(t) := f(t,0)= 
       \left( t,\ 
              \frac{t^3}{6} ,\ 
              \frac{t^2}{2} + \frac{t^3}{6} \right).
\]
We remark that the curvature function $\kappa$ and
torsion function $\tau$ of $\hat\gamma$ is given by 
\[
   \kappa(t)=\sqrt{\frac{2\delta}{(2+2t^2+2t^3+t^4)^3}},\quad
   \tau(t)=-\frac{4}{\delta}
        \qquad (\delta:=4+8t^2+8t^3+t^4).
\]
In particular, $\hat \gamma$ is non-symmetric and non-planar
as the singular curve of $f$. 
We set
$g := (g^1, g^2, g^3)$,
where
\begin{align*}
 g^1(u,&v)
   := u+\frac{u^2 v^2}{2}-\frac{u^3 v^2}{2}-\frac{u v^4}{2}
       +\frac{v^5}{30}+\frac{u^3 v^3}{6}
 +\frac{9 u^2 v^4}{4}       +\frac{v^6}{6},\\
 g^2(u,&v)
   := \frac{v^2}{2}+\frac{u^3}{6}-u^2 v^2+\frac{u v^3}{3}
        +2 u^3 v^2-\frac{u^2 v^3}{3}+u v^4-\frac{v^5}{5}
        -\frac{9 u^4 v^2}{4}-6 u^2 v^4\\
       &+\frac{13 u v^5}{15}-\frac{11 v^6}{36},\\
 g^3(u,&v)
   := \frac{u^2}{2}+\frac{u^3}{6}-u v^2+\frac{v^3}{6}+u^2 v^2
        +\frac{v^4}{2}-\frac{u^2 v^3}{3}-\frac{5 u v^4}{2}
        -\frac{v^5}{15}-2 u^4 v^2\\
       &+\frac{2 u^3 v^3}{3}+6 u^2 v^4+\frac{2 u v^5}{5}
        +\frac{25 v^6}{18}.
\end{align*}
The singular set of $g$ consists of generic cuspidal edges, and
coincides with that of $f$, namely
$g(t,0)=\hat\gamma(t)$ holds.
This $g$ gives an approximation of  $\check f$.
In fact, one can easily check that
the coefficients of
the first fundamental form of $g$
coincide with those of $f$
up to the fifth-order terms of $u,v$ near the origin 
(see Figure \ref{fig:isomer_ex}).

\section{Realization of generic intrinsic cuspidal edges into $\R^3$.}
\label{sec:realize}

Let $d\sigma^2$ be a positive semi-definite {\it real analytic\/}
metric defined on a neighborhood of the origin in the $uv$-plane $U$.
Then $d\sigma^2$ can be written in the following form:
\[
   d\sigma^2=E\,du^2+2\,F\,du\,dv+G\,dv^2.
\]
The metric is called a (germ of) {\it Kossowski's metric\/}
 (cf.\ \cite{HHNSUY})
if it satisfies the following conditions:
\begin{enumerate}
 \item the $u$-axis consists of a singular set of $d\sigma^2$,
 \item $E_v=G_v=0$ along the $u$-axis,
 \item there exists a real analytic function
       $\lambda$ on $U$ such that $EG-F^2=\lambda^2$, and
 \item the gradient vector field $\nabla \lambda=(\lambda_u,\lambda_v)$
       does not vanish along the $u$-axis.
\end{enumerate}
Further systematic treatments of Kossowski's metrics are
given in \cite{HHNSUY}.
Let $K$ be the Gaussian curvature of $d\sigma^2$
on $U\setminus\{v= 0\}$.
Kossowski showed in \cite{K} that 
\[
  K\,d\hat A \qquad (d\hat A:=\lambda\, du\wedge dv)
\]
can be smoothly extended on $U$
(cf.\ the condition (d) in the introduction), 
and proved the following assertion.

\begin{fact}[Kossowski]
 Suppose that $K\,d\hat A$ does not have zeros on the $u$-axis.
 Then there exist a neighborhood $V(\subset U)$ and a
 real analytic wave front 
 $f:V\to \R^3$ such that the pull-back metric of the 
 canonical metric of $\R^3$ by $f$ coincides with $d\sigma^2$.
\end{fact}

See \cite{K} and \cite{HHNSUY} for detailed discussions.  
If the null-direction of the metric $d\sigma^2$
is transversal to the $u$-axis, the singular points
of $d\sigma^2$ are called of {\it $A_2$-singularities}
or {\it intrinsic cuspidal edges}.
The induced metrics of wave fronts in $\R^3$
are all considered as Kossowski's metrics
and cuspidal edge singular points corresponds to
$A_2$-singularities (cf.\ \cite{HHNSUY}). 
Moreover, in \cite{HHNSUY}, the following expression of the singular curvature
is given;
\[
   \kappa_s:=\frac{-F_vE_u+2EF_{uv}-EE_{vv}}{2E^{3/2}\lambda_v}
\]
under the assumption that $\lambda_v>0$.
Since this expression of $\kappa_s$ does not depend on a
choice of such a local coordinate $(u,v)$, it can be
considered as an invariant of the metric $d\sigma^2$
at $A_2$-singularities.
We can prove the following assertion as an
modification of the proof of Theorem~\ref{thm:A}:
\begin{theorem}
 Let $d\sigma^2$ be a real analytic 
 Kossowski's metric given as above
 and $\kappa_s(t)$ the singular curvature
 function along the $u$-axis.
 Let $\sigma(t)$ be a real analytic regular space curve
 whose curvature function $\tilde\kappa(t)$
 satisfies 
 \begin{equation}\label{eq:A2}
  \tilde\kappa(t)>|\kappa_s(t)|
 \end{equation}
 for all sufficiently small $t$.
 Suppose that $K\,d\hat A$ does not have zeros on the $u$-axis.
 Then there exist a neighborhood $V(\subset U)$ and a
 real analytic wave front 
 $f:V\to \R^3$ such that
 \begin{enumerate}
  \item the first fundamental form of $f$ coincides 
	with $d\sigma^2$, 
  \item $f(t,0)=\sigma(t)$ holds for each $t$.
 \end{enumerate}
\end{theorem}

In \cite{K}, the isometric deformations of 
singularities are not discussed, 
and Theorem~\ref{thm:A} can be considered as 
a refinement of \cite[Theorem 1]{K} in the case
of cuspidal edge singularities.
As pointed out in the introduction,
the following proof is a different from Kossowski's original approach.

\begin{proof}
 Let 
 $d\sigma^2=E\,du^2+2\,F\,du\,dv+G\,dv^2$ 
 be a Kossowski's metric
 such that the singular set $\{v=0\}$ consists of 
 $A_2$-singularities.
 Then, without loss of generality, we may assume that  
$d\sigma^2$
 satisfies the following expressions (cf.\ \cite{HHNSUY});
 \begin{equation}
   E=1+v^2E_0,\quad
   F=0, \quad
   G=v^2 G_0,
 \end{equation}
where $E_0=E_0(u,v)$ and $G_0=G_0(u,v)$ are certain real analytic functions.
 We now suppose that there exists a real analytic
 wave front $f:U\to \R^3$ so that the first fundamental form
 of $f$ is equal to $d\sigma^2$.
 We can define a real analytic map $\phi=\phi(u,v)$ so that $f_v=v \phi$.
 Then, it holds that $G_0=\phi\cdot \phi$.
 Keeping \eqref{e:4} in mind, we have that
 \begin{equation}\label{e:4rev}
  \phi_v\cdot \phi=\frac{(\phi\cdot \phi)_v}2=\frac12 (G_0)_v.
 \end{equation}
 On the other hand, since
 \[
   v f_u\cdot \phi=f_u\cdot f_v=F=0
 \]
 we have $f_u\cdot \phi=0$.
 In particular, we get
 (cf.\ \eqref{e:5})
 \begin{equation}\label{e:5rev}
  \phi_v\cdot f_u=(\phi\cdot f_u)_v-\phi\cdot f_{uv}
   =-\phi\cdot f_{uv}
   =-\phi\cdot (f_{v})_u
   =-\frac{v(G_0)_u}2. 
 \end{equation}
 Next \eqref{e:6} in mind, we have that
 \begin{equation}\label{e:6rev}
  \phi\cdot f_{uu}
   =\frac{(f_v\cdot f_{u})_u-(f_{uv}\cdot f_u)}v
   =-\frac{(f_{uv}\cdot f_u)}v
   =-\frac{E_v}{2v}=-\frac{2E_0+v (E_0)_v}{2}.
 \end{equation}
 By \eqref{e:4rev}, \eqref{e:5rev} and \eqref{e:6rev},
 $f$ must satisfy the equation
 \begin{equation}\label{pde:new3}
  (f_v,r_v,\phi_v)=\tilde{\F}(u,v;\phi_u,(\phi,f_u,r_u)),
 \end{equation}
 where $r:=f_u$, and $\tilde{\F}:=(\tilde \F^1,\tilde \F^2,\tilde \F^3)$ is given by
 \begin{align*}
  &\tilde \F^1(u,v;\, {x},({y}_1,{y}_2,{y}_3)):= 
      v {y}_1,\\
  &\tilde \F^2(u,v;\, {x},({y}_1,{y}_2,{y}_3)):= v {x},\\
  &\tilde\F^3(u,v;\, {x},({y}_1,{y}_2,{y}_3))\\
   &:=
   \frac12((y_1,y_2,y_3)^T)^{-1}
   \pmt{(G_0)_v\\ 
          - v (G_0)_u\\
          - 3 (E_0)_v - v (E_0)_{vv}-v (G_0)_{uu}
               +2v (x\cdot x)
   }.
 \end{align*}
 Applying the Cauchy-Kovalevskaya theorem,
 we can get a real analytic solution of
 the equation \eqref{pde:new3} under the same initial 
 conditions as in  the proof of Theorem~\ref{thm:A}.
 Since the product curvature $\kappa_\Pi$ can be reformulated
 as an invariant of the $A_2$-singular point of
 a given Kossowski metric as shown in \cite{HHNSUY},
 and the condition $K\,d\hat A\ne 0$ implies the
 condition $\kappa_\Pi\ne 0$.
 Thus we can prove that $f$ has cuspidal edge singularity
 along the $u$-axis, and we get the assertion.
\end{proof}

\section*{Appendix}
The following assertion is applied to prove
Proposition \ref{prop:ini}:
\begin{lemma*}
 Let $S^2$ be the unit sphere in $\R^3$  centered at the origin.
 Let $a$, $b\in S^2$ be two
 mutually orthogonal unit vectors, and $\mu$ a real number
 with $|\mu|<1$.
 Then there exists a unit vector $w\in S^2$
 satisfying
 \[
    w \cdot a=0,\qquad
    w \cdot b=\mu.
 \]
 Moreover, such a vector $w$ is uniquely determined
 under the assumption that the determinant
 $\det(a,b,w)$ 
 is positive $($resp.\ negative$)$.
\end{lemma*}
\begin{proof}
 Let $M:=(a,b,a\times b)\in\SO(3)$,
 where ``$\times$'' is the vector product of $\R^3$.
 Then the vector 
 \[
   w:=M\begin{pmatrix}
	 0 \\
	 \mu \smallskip\\
	 \pm\sqrt{1-\mu^2}
       \end{pmatrix}
 \]
 has the desired property.
 The uniqueness can be shown immediately.
\end{proof}
 
\section*{Acknowledgments}
The authors are grateful to Wayne Rossman
for valuable comments.


\begin{thebibliography}{20}
\bibitem{dt}
 \textsc{F. S. Dias and F. Tari},
 On the geometry of the cross-cap in the Minkowski $3$-space, 
 preprint, 2012.\newline
 Available from www2.icmc.usp.br/{$^\sim$}faridtari/Papers/DiasTari.pdf.

\bibitem{FT} 
 \textsc{D. Fuchs and S. Tabachnikov},
 Thirty lectures on Classic Mathematics,
 American Mathematical Society, 
 Providence, Rhode Island, 2007.

\bibitem{fh}
 \textsc{T. Fukui and M. Hasegawa},
 The Fronts of Whitney umbrella---a differential
 geometric approach via blowing up,
 J.\ Singul., 4 (2012), 35--67.


\bibitem{ggs}
 \textsc{R. Garcia, C. Gutierrez, and J. Sotomayor},
 Lines of principal curvature around umbilics and Whitney
 umbrellas,
 Tohoku Math.\ J., 52 (2000), 163--172.

\bibitem{HHNUY}
  \textsc{
  M. Hasegawa, A. Honda, K. Naokawa, M. Umehara, and K. Yamada},
  Intrinsic invariants of cross caps,
  Selecta Mathematica, {\bf 20}  (2014), 769-785.

\bibitem{HHNSUY}
 \textsc{M. Hasegawa, 
 A. Honda, K. Naokawa, K. Saji, M. Umehara, 
 K. Yamada},
 Intrinsic properties of singularities of surfaces,
 in preparation.

\bibitem{K}
 \textsc{M. Kossowski}, 
 Realizing a singular first fundamental form as a 
    nonimmersed surface in Euclidean 3-space, 
  J. Geom. 81 (2004), 101--113.

\bibitem{KRSUY}
  \textsc{%
  M. Kokubu,  W. Rossman, K. Saji, M. Umehara, and K. Yamada},
  Singularities of flat fronts in hyperbolic $3$-space,
  Pacific J. Math., 221 (2005), 303--351.
\bibitem{MB}
  \textsc{
  L. F. Martins and J. J. Nu\~no-Ballesteros},
  Contact properties of surfaces in $\R^3$ with corank $1$
       singularities,  
  preprint, 2012.\newline
  Available from
  www.uv.es/nuno/Preprints/Nuno$\_$Martins.pdf.
\bibitem{MS}
  \textsc{ 	
  L. F. Martins and K. Saji},
  Geometric invariants of cuspidal edges,
         preprint, 2013.\newline
       Available from
      www.ibilce.unesp.br/Home/Departamentos/%
      Matematica/Singularidades/martins-saji-geometric.pdf

\bibitem{MSUY}
  \textsc{
  L. F. Martins, K. Saji, M. Umehara, and K. Yamada},
  Behavior of Gaussian curvature 
  near non-degenerate
  singular points on wave fronts, 
  preprint,
  2013, arXiv:1308.2136.

\bibitem{MU}
 \textsc{
 S. Murata and M. Umehara},
 Flat surfaces with singularities in Euclidean 3-space, 
 J. Differential Geometry  82 (2009), 279--316．

\bibitem{Oset-Tari}
 \textsc{
 R. Oset Sinha and  F. Tari},
 Projections of surfaces in $\mathbb{R}^4$
     to $\mathbb{R}^3$ and the geometry of their singular images,
 preprint, 2012.\newline
  Available from
  www2.icmc.usp.br/{$^\sim$}faridtari/Papers/OsetTariSingularSurfaces.pdf.
\bibitem{Sp} 
  \textsc{M. Spivak},
  A comprehensive Introduction to Differential Geometry V,
  Publish or Perish Inc. Houston, Texas, 1999.
\bibitem{SUY}
  \textsc{K. Saji, M. Umehara, and K. Yamada},
  The geometry of fronts,
   Ann.\ of Math., 169 (2009), 491--529.
\bibitem{faridtari}
  \textsc{F. Tari},
  On pairs of geometric foliations on a cross-cap,
  Tohoku Math. J. 59 (2007), 233--258.
 \end{thebibliography}
\end{document}